\documentclass[11pt]{amsart}
\usepackage{amsthm,amssymb,amsfonts,amsmath,amstext,amscd}
\usepackage[margin=1.2in]{geometry} 
\usepackage{graphicx}
\usepackage{comment}
\usepackage{latexsym}
\usepackage{enumerate}
\usepackage{accents}
\usepackage{cancel}
\usepackage{eucal}
\usepackage{hyperref}

\newtheorem{theorem}{Theorem}
\newtheorem{lemma}[theorem]{Lemma}

\newtheorem{proposition}[theorem]{Proposition}
\newtheorem{Conditions}[theorem]{Conditions}

\newcommand{\R}{\mathbb{R}}

\DeclareMathOperator{\spt}{spt}

\begin{document}

\title[Local convexity estimates for fully nonlinear curvature flows]{Local umbilic, convexity and cylindrical estimates for fully nonlinear curvature flows}

\author{Mat Langford}
\email{mathew.langford@anu.edu.au}
\address{Mathematical Sciences Institute\\
Australian National University\\
Canberra\\
ACT 0200\\
Australia}

\author{James McCoy}
\address{School of Information and Physical Sciences\\
University of Newcastle\\
University Drive\\
Callaghan\\
NSW 2308\\
Australia}
\email{James.McCoy@newcastle.edu.au}

\date{\today}

\begin{abstract}
In a recent article \cite{MR4609824}, a localization of the Huisken--Stampacchia iteration method was developed, and used to establish localizations of the well-known ``umbilic'', ``convexity'' and ``cylindrical'' estimates for hypersurfaces evolving in Euclidean space by mean curvature flow \cite{MR0772132,MR1666878,MR1719551,MR2461428}. Here, we adapt the methods of \cite{MR4609824} to treat more general (fully nonlinear) flows, establishing localizations of asymptotically sharp curvature pinching estimates for hypersurfaces evolving by one-homogeneous functions of curvature under very general conditions. 
We also briefly describe how the method can be adapted to treat the deformation of hypersurfaces in curved ambient spaces (by suitable speed functions), which is fundamental for many important applications of such flows. 
\end{abstract}


\thanks{The research of the first author was supported by DECRA fellowship DE200101834 of the Australian Research Council (ARC).  The research of the second author was supported by ARC Discovery Project DP180100431.}

\maketitle

\setcounter{tocdepth}{1}
\tableofcontents

\section{Introduction} \label{S:intro}
\newtheorem{Fconditions}{Conditions}[section]
\newtheorem{main}[Fconditions]{Theorem}

Extrinsic geometric flows of are diffusive processes which drive interfaces in their normal direction with speed determined pointwise by curvature. Such equations arise naturally in various physical processes \cite{MR0078836,vonNeumann52,Firey} and have led to a number of remarkable applications in geometry and topology \cite{MR1957035,MR3714512,MR1267897,MR2461428}
. The most well-known example of such a process is the (quasilinear) mean curvature flow \cite{MR0772132}
, but many other (fully nonlinear) examples have been studied (see, for example, \cite{MR2670672,MR1385524,MR0826427,MR0862712,MR2729317,MR3070558,MR3338439,MR2198805}), not least because specific applications often lead naturally to, or even necessitate, the consideration of specific speeds \cite{MR1957035}. One well-known example of this phenomenon is the theorem of Andrews \cite{MR1267897}, which demonstrates that a closed hypersurface of a compact Riemannian manifold whose principal curvatures $\kappa_i$ exceed the square root $\sqrt{K}$ of a non-positive lower bound $-K$ for the ambient sectional curvatures deforms smoothly to a round point when evolved by the harmonic mean of the shifted principal curvatures, $\kappa_i-\sqrt{K}$. One nice application of this result is a quantitative ``pinched sphere'' theorem (which cannot be established using the mean curvature flow). 

The key difficulty in the analysis and application of such flows is the formation of finite time singularities. One powerful and robust tool for analysing singularity formation in extrinsic geometric flows is Huisken--Stampacchia iteration, which is the main tool in the proof of various important \emph{a priori} curvature pinching estimates that improve at the onset of singularities (see, for example, \cite{MR0772132, MR1666878, MR1719551,MR4609824,Ng15,LangfordNguyen,LynchNguyen} for the mean curvature flow and \cite{MR0826427, MR0862712, MR3265960, MR3299822, MR3218814, MR3714512, MR4484205} for fully nonlinear flows). In a recent article \cite{MR4609824}, the Huisken--Stampacchia scheme was localized in space in the context of mean curvature flow of hypersurfaces in Euclidean space, leading to local versions of all of the known asymptotic curvature pinching estimates in that setting. Our goal in this paper is to extend this method to more general fully nonlinear speeds and general ambient spaces, both of which are crucial in the context of geometric applications. Our result is the localization, in general ambient spaces, of essentially all known asymptotic curvature pinching estimates for fully nonlinear one-homogeneous flows (i.e. \cite[Theorem 5.1]{MR0826427}, \cite[Theorem 5.1]{MR0862712}, \cite[Theorem 1.2]{MR3299822}, \cite[Theorem 1.1]{MR3218814}, \cite[Theorem 1.3]{MR3265960}, \cite[Theorem 3.1]{MR3714512}, \cite[Theorem 1.1]{MR4129354}, and \cite[Theorem 1.1]{MR4484205}). 

The structure of this article is as follows. In Section \ref{S:prelim} we provide a precise definition of the class of flow speeds in which we are interested, and recall some properties of curvature functions of interest. Here we also set up notation, describe our cutoff functions, and recall a general Poincar\'{e}-type inequality, the Sobolev inequality of Michael and Simon, and Stampacchia's lemma, which are key tools needed for Huisken--Stampaccia iteration. In Section \ref{S:Fconvex}, we localize the convexity \cite{MR3299822} and cylindrical \cite{MR3265960} estimates for flows by convex speeds. The arguments presented here will form the template for later sections. In Section \ref{S:surfaces}
, we localize the convexity and umbilic estimates for the flow of surfaces \cite{MR2729317,MR3299822} 
by any admissible speed. In Sections \ref{S:concave} and \ref{S:Fcic}, we consider flows of hypersurfaces by concave speeds $f:\Gamma^n\to \R$, $\Gamma_+^n\subset \Gamma^n$, establishing local cylindrical estimates when $f|_{\partial\Gamma^n}=0$ and a local umbilic estimate when $f$ is inverse-concave $\Gamma^n=\Gamma_+^n$ \cite{MR4129354}. We also establish a local convexity estimate for flows of hypersurfaces by the special class of concave speeds introduced by Lynch \cite{MR4484205}. 
Finally, in Section \ref{S:Riem}, we show how our localized estimates may be established when the ambient space is not flat.

\section{Preliminaries} \label{S:prelim}

\subsection{Curvature functions and admissible flow speeds} We refer to a smooth function $f:\Gamma^n\to\R$ defined on some convex cone $\Gamma^n\subset\R^n$ which contains the positive cone $\Gamma_+^n$ as an \emph{admissible} speed function provided the following conditions hold.


\begin{Conditions} \label{conds:admissibility} 
\mbox{}
\begin{enumerate}[(i)]
\item $f$ is symmetric under permutations of its arguments.
\item $f$ is strictly increasing in each argument. 
 \item $f$ is homogeneous of degree $1$.
\item $f(1,\dots,1)=n$.
 \end{enumerate}
\end{Conditions}

The final normalization condition is merely for convenience, and does not diminish the generality of our results (it can be arranged by simply rescaling the time parameter).

Given an admissible speed $f$, a one-parameter family $X:M^n\times I\to \R^{n+1}$ of hypersurface immersions $X(\cdot,t):M^n\to\R^{n+1}$ \emph{evolves by the corresponding curvature flow} if
\begin{equation} \label{E:theflow}
  \frac{\partial X}{\partial t} \left( x, t\right) = - 
  F(x,t)\nu\left( x, t\right) \mbox{,}
\end{equation}
where $\nu$ is a local choice of unit normal vector field and, denoting the principal curvatures (eigenvalues of the shape operator $A:= D\nu$) by $\kappa_1,\dots,\kappa_n$, the scalar $F$ is defined by $F:=f(\kappa_1,\dots,\kappa_n)$.

Aside from the freedom to reparametrise, which can be removed by fixing a particular choice of parametrisation, flows by such speeds are parabolic and thus enjoy nice properties such as uniqueness and short time existence but, since they are fully nonlinear, can still behave quite badly in the absence of additional conditions in dimensions $n\ge 3$ (e.g. losing regularity or convexity properties \cite{MR3070558}). For this reason, various mild concavity conditions for $f$ are typically imposed in dimensions $n\ge 3$. 

Unless $n=2$, 
we shall be interested in flows by admissible speeds which satisfy one of the following additional conditions.

\begin{Conditions}\label{conds:concavity}
\mbox{}
\begin{enumerate}[(i)]
\item $f$ is convex.
\item $f$ is concave and vanishes on\footnote{Note that any admissible speed function admits a continuous extension to the closure $\overline\Gamma{}^n$ of $\Gamma^n$ (cf. \cite[Lemma 1]{MR3070558}).} $\partial\Gamma^n$.
\item $\Gamma^n=\Gamma_+^n$ and both $f$ and the dual function $f_\ast$ defined by\\ 
$f_\ast(r_1,\dots,r_n):=f(r_1^{-1},\dots,r_n^{-1})^{-1}$ are concave.
\end{enumerate}
\end{Conditions}


There are many non-trivial examples (defined on non-trivial cones) of speed functions satisfying Conditions \ref{conds:admissibility} and one or more of the Conditions \ref{conds:concavity}. See, for example, \cite{MR4249616,MR3218814,MR3299822,MR3070558}.

Any smooth, symmetric function $f:\Gamma^n\subset\R^n\to\R$ gives rise to a well-defined smooth function on the set of symmetric $n\times n$ matrices with eigenvalues in $\Gamma^n$ which is invariant under conjugation by orthogonal matrices (and vice versa) \cite{Glaeser63}. We shall denote both of these functions using the same letter (in this case $f$). First and second partial derivatives of these functions (with respect to either eigenvalue or matrix variables) will be denoted $\dot f^k$, $\ddot f^{pq}$ and $\dot{f}^{kl}$, $\ddot f^{pq,rs}$, respectively. That is,
\[
\left.\tfrac{d}{ds}\right|_{s=0}f\left(z+sv\right)= \dot f^{k} \left(z\right)v_{k}\,,\;\;
\left.\tfrac{d^2}{ds^2}\right|_{s=0}f\left(z+sv\right) = \ddot f^{pq} \left(z\right)v_pv_q
\]
and
\[
\left.\tfrac{d}{ds}\right|_{s=0}f\left(Z+sV\right) = \dot f^{kl} \left(Z\right)V_{kl}\,,\;\; \left.\tfrac{d^2}{ds^2}\right|_{s=0}f\left( Z+sV\right)=\ddot f^{pq,rs}\left(Z\right)V_{pq}V_{rs}\mbox{.}
\]
When evaluating these derivatives at the principal curvatures/component matrix of the second fundamental form with respect to an orthonormal basis along an evolving hypersurface, we use the corresponding capital letter. E.g. $\dot f^k(\kappa_1,\dots,\kappa_n)=\dot F^k$.

When $f$ is an admissible speed, the derivative $\dot f$ is positive definite. It therefore induces a natural elliptic operator $\mathcal{L}:=\dot F^{kl} \nabla_{k} \nabla_{l}$, an inner product, $\left< v, w\right>_F := \dot F^{ij} v_i w_j$, and a norm $\left| v\right|^2_F := \dot F^{ij} v_i v_j$ along any solution to the corresponding flow. We denote by $\left< u, v\right>$ and $\left| v \right|^2$ (without the subscript) the inner product and norm with respect to the induced evolving metric $g$.

\subsection{Cut-off functions}

As in \cite{MR4609824}, for localization we introduce smooth cut-off functions $\zeta \in C_0^\infty(\R^{n+1})$ localized in ambient balls $B_\lambda\subset\R^{n+1}$ of radius $\lambda$, and we set $\psi := \zeta \circ X$. 
%
A short computation yields
\begin{equation} \label{E:evlnpsi}
  \left( \partial_t - \mathcal{L} \right) \psi = -\dot F^{ij} D^2 \zeta\left(D_iX, D_jX \right) = - \dot F^{ij} D_{i}D_j \zeta \mbox{,}
  \end{equation}
where for the last step we have employed a local orthonormal frame for $\R^{n+1}$ adapted to the hypersurface.

For fixed $r$ and $R$ satisfying $0<r<R\leq \lambda$, we require that $\zeta : \mathbb{R}^{n+1} \rightarrow [0,\infty)$ satisfy:
\begin{enumerate}
  \item[\textnormal{1.}] $\zeta\left( X\right) = 0$ for $X\not \in B_R$ and $\zeta\left( X \right) =1$ for $X \in B_r$;
  \item[\textnormal{2.}] $\left| D_i \zeta \right|^2 \leq 10 \left( R-r\right)^{-2} \zeta$ for each $i$;
  \item[\textnormal{3.}] $\left| D_i D_j \zeta \right|^2 \leq 10 \left( R-r\right)^{-2}$ for each $i$ and $j$.
\end{enumerate}

\subsection{A Poincar\'e-type inequality}

Denote by 
$$\mathrm{Cyl} := \cup_{m=0}^{n-1} \mathrm{Cyl}_m$$
the set of \emph{cylindrical points}, where $\mbox{Cyl}_m$ is the set of \emph{$m$-cylindrical points}; that is, those points $\R^n$ with $m$ coordinates equal to zero and $n-m$ coordinates positive and equal. We will make use of the following inequality (which generalizes an idea from \cite{MR3714512}).

\begin{proposition}[Poincar\'e-type inequality {\cite[Proposition 2.7]{MR3669776}}]\label{prop:Poincare}
Let $\Gamma^n \subset \R^n$ be an open, symmetric cone satisfying $\overline{\Gamma}{}^n\backslash \left\{ 0 \right\} \subset \left\{w \in \R^n: w_1+\dots+w_n>0\right\}$ and $\overline{\Gamma}{}^n\cap \mathrm{Cyl} = \emptyset$. There is a constant $\gamma = \gamma(n,\Gamma^n) > 0$ with the following property.  Let $X: M^n \rightarrow \mathbb{R}^{n+1}$ be a smooth hypersurface and $u\in W^{2, 1}( M^n)$ a function for which the set\\
$\left\{(\kappa_1(x),\dots,\kappa_n(x)): u(x)>0, x \in M^n \right\}$ is precompact and lies in $\Gamma^n$. For every $r>0$,
$$\gamma \int u^2 \left| A\right|^2 d\mu \leq r^{-1} \int \left| \nabla u \right|^2 d\mu + \left( 1 + r\right) \int u^2 \frac{\left| \nabla A \right|^2}{H^2} d\mu \mbox{.}$$
Here $H$ denotes the mean curvature.
\end{proposition}

\subsection{The Sobolev inequality of Michael and Simon}

We shall need the following Sobolev inequality, which is an immediate consequence of the celebrated inequality of Michael and Simon and H\"older's inequality.

\begin{proposition}[Michael--Simon Sobolev inequality \cite{MR0344978}]
Let $X:M^n\to\R^{n+1}$, $n\ge 3$, be a smoothly immersed hypersurface. If $u\in H^1_0(M^n)$, then
\begin{equation}\label{eq:MichaelSimon}
\left(\int u^{2^\ast}d\mu\right)^{\frac{1}{2^\ast}}\le C\left(\int\left(\vert\nabla u\vert^2+H^2u^2\right)d\mu\right)^{\frac{1}{2}}\,,
\end{equation}
where $\frac{1}{2^\ast}=\frac{1}{2}-\frac{1}{n}$ and $C=C(n)$. 
\end{proposition}

Unfortunately, this only applies when $n\ge 3$. When $n=2$, we instead 
make use of the following Poincar\'e inequality.

\begin{proposition}[Poincar\'e inequality] \label{T:Poincare}
Let $X:M^2\to\R^{3}$ be a smoothly immersed surface. If $u\in H^1_0(M^n)$ and $q\ge 1$, then
\begin{equation}\label{eq:classicalPoincare}
\left(\frac{1}{\vert{\spt u}\vert}\int \vert u\vert^{2q}d\mu\right)^{\frac{1}{2q}}\le C(q^2+1)\left(\int\big(\vert\nabla u\vert^2+H^2u^2\big)d\mu\right)^{\frac{1}{2}}\,,
\end{equation}
where $C=C(n)$ is the Sobolev constant.
\end{proposition}

\subsection{Stampacchia's lemma}

Proving each of our results requires an application of Stampacchia's lemma, which we now recall.

\begin{lemma}[Stampacchia's lemma \cite{MR0251373, MR1786735}] \label{T:Stampacchia}
Let $\varphi: \left[ k_0, \infty \right) \rightarrow \mathbb{R}$ be a non-negative, nonincreasing function. If
\begin{equation}
\varphi\left( h \right) \leq \frac{C}{\left( h-k\right)^\alpha }\varphi\left( k \right)^\beta
\end{equation}
when $h>k>k_0$ for some constants $C>0$, $\alpha>0$ and $\beta>1$, then
\[
\varphi\left( k_0 + d\right) = 0 \mbox{,}
\]
where $d^\alpha = C\, \varphi\left( k_0 \right)^{\beta-1} 2^{\alpha\beta / \left( \beta -1\right)}$.
\end{lemma}

\section{Hypersurfaces evolving by convex speeds} \label{S:Fconvex}

We first consider flows of hypersurfaces by \emph{convex} admissible speeds $f:\Gamma^n\to\mathbb{R}$. A key feature of such flows is that they preserve any convex cone of curvatures which contains the positive cone, and hence remain uniformly parabolic throughout the evolution \cite[Lemma 2.4]{MR3218814}. 
That is, if, for some convex symmetric cone $\Gamma^n_0\subset\R^n$ satisfying $\Gamma^n_+\subset \Gamma^n_0$, the principal curvatures satisfy $\vec\kappa:=(\kappa_1,\dots,\kappa_n)\in \Gamma^n_0$ at the parabolic boundary of a region in which the flow is properly defined, then $\vec\kappa\in \Gamma^n_0$ in the interior of this region as well. If, in addition, $\Gamma^n_0\Subset\Gamma^n$ (by which we mean that $\overline{\Gamma}{}^n_0\cap S^n\subset\Gamma^n$ or, equivalently, $\overline\Gamma{}^n_0\setminus\{0\}\subset\Gamma^n$), then, since $\dot f$ is zero-homogeneous, there exists some $\Lambda=\Lambda(n,f,\Gamma^n_0)<\infty$ such that
\begin{equation}\label{eq:uniform parabolic}
\Lambda^{-1}\delta^{ij}\le \dot F^{ij}\le \Lambda\delta^{ij}\,.
\end{equation}
Note that, for a compact hypersurface with $\vec\kappa\in \Gamma^n$ at all points, some such cone $\Gamma^n_0$ can always be found when $\Gamma^n$ is convex and contains $\Gamma^n_+$.

The convexity estimate \cite[Theorem 1.1]{MR3218814} admits the following localization.

\begin{theorem}[Convexity estimate --- convex speeds]\label{thm:convexityc}
Let $f:\Gamma^n\to\R$ be a convex admissible speed function with $\Gamma^n_+\subset\Gamma^n$. If a solution to the flow with speed $F=f(\vec\kappa)$ is properly defined in the spacetime cylinder $\overline B_{\lambda R} \times \left[0, \frac{1}{2n} R^2 \right) \subset \mathbb{R}^{n+1} \times \mathbb{R}$ and satisfies
\begin{enumerate}
\item $\vec\kappa(x,t)\in \Gamma^n_0$ for all $(x,t)\in B_{\lambda R} \times \left\{ 0 \right\} \cup \partial B_{\lambda R} \times \left( 0, \frac{1}{2n} R^2 \right)$ for some convex cone $\Gamma^n_0\Subset\Gamma^n$.
\item $\displaystyle\sup_{B_{\lambda R} \times \left\{ 0 \right\} \cup B_{\lambda R}\setminus B_{(\lambda-\delta)R} \times \left( 0, \frac{1}{2n} R^2 \right)} F \leq \Theta\, R^{-1}$, and
\item $\displaystyle\frac{R^2}{2n} \int_{B_{\lambda R}} F^q d\mu_0 + \frac{1}{\delta^2} \int_0^{\frac{R^2}{2n}}\!\!\!\int_{B_{\lambda R} \backslash B_{(\lambda-\delta)R}} F^q d\mu\, dt \leq \left( \Theta\, R^{-1}\right)^q \left( V \, R\right)^{n+2}$,
\end{enumerate}
then, given any $\varepsilon>0$ and $\vartheta \in \left(0,1\right)$, it satisfies
\begin{equation}\label{eq:convexity convex speeds}
\kappa_1 \geq -\varepsilon\,F - C_\varepsilon R^{-1}\;\;\text{in}\;\; B_{\left( \lambda- \delta\right) \vartheta R} \times \left( 0, \tfrac{1}{2n} R^2 \right)\,,
\end{equation}
where $C_\varepsilon= c\left(n,f,\Gamma^n_0,q,\varepsilon \right) \Theta \left(\frac{\Theta\, V}{1-\vartheta} \right)^{\frac{2}{q}}$.
\end{theorem}

Consider now the $(m+1)$-positive cone
\[
\Gamma^n_{m+1}:=\{z\in\mathbb{R}^n:z_{\sigma(1)}+\dots+z_{\sigma(m+1)}>0\;\;\text{for all}\;\; \sigma\in P_n\}\,.
\]
Note that, since any $z\in\R^n$ may be represented as $z=z_0+\lambda\vec1$ for some $\lambda\in\R$ and $z_0\in \vec1{}^\perp=\{z\in\R^n:z_1+\dots+z_n=0\}$, any convex admissible speed function $f:\Gamma^n\to\R$ satisfies
\[
f(z)=f(z_0+\lambda\vec 1)\ge f(\lambda\vec 1)+Df|_{\lambda\vec 1}\cdot z_0=f(\lambda\vec 1)=\lambda f(\vec 1)=n\lambda=f_0(\lambda\vec 1)\,,
\]
where $f_0(z):=z_1+\dots+z_n$. In particular, $f>0$ on $\overline\Gamma{}^n_{m+1}\setminus\{0\}$ if $m\le n-2$.

Since (by the tensor maximum principle) flows of hypersurfaces by convex speeds preserve cones of the form\footnote{These cones are convex but do not contain the positive cone.} \cite{MR3265960}
\[
\Gamma^n_0:=\{\kappa_1+\dots+\kappa_{m+1}\ge \alpha F\}\,,\;\;\alpha>0\,,
\]
we are able to establish the following localization of the cylindrical estimates \cite[Theorem 1.3]{MR3265960}.

\begin{theorem}[Cylindrical estimates --- convex speeds]\label{thm:cylindrical convex}
Given $m\in\{0,\dots,n-2\}$, let $f:\Gamma^n\to\R$ be a convex admissible speed function with $\Gamma_{m+1}^n\subset \Gamma^n$. If a solution to the flow with speed $F=f(\vec\kappa)$ is properly defined in the spacetime cylinder $\overline B_{\lambda R} \times \left[ 0, \frac{1}{2n} R^2 \right) \subset \mathbb{R}^{n+1} \times \mathbb{R}$ and satisfies
\begin{enumerate}
\item $\displaystyle\inf_{B_{\lambda R} \times \left\{ 0 \right\} \cup \partial B_{\lambda R} \times \left( 0, \frac{1}{2n} R^2 \right)} \frac{\kappa_1 + \cdots + \kappa_{m+1}}{F} \geq \alpha > 0$,
\item $\displaystyle\sup_{B_{\lambda R} \times \left\{ 0 \right\} \cup B_{\lambda R}\setminus B_{(\lambda-\delta)R} \times \left( 0, \frac{1}{2n} R^2 \right)} F \leq \Theta\, R^{-1}$, and
\item $\displaystyle\frac{R^2}{2n} \int_{B_{\lambda R}} F^q d\mu_0 + \frac{1}{\delta^2} \int_0^{\frac{R^2}{2n}}\!\!\!\int_{B_{\lambda R} \backslash B_{(\lambda-\delta)R}} F^q d\mu\, dt \leq \left( \Theta\, R^{-1}\right)^q \left( V \, R\right)^{n+2}$,
\end{enumerate}
then, given any $\varepsilon>0$ and $\vartheta \in \left(0,1\right)$, it satisfies
\begin{equation}\label{eq:cylindrical convex speeds}
\kappa_1+\dots+\kappa_{m+1}-c_mF\geq -\varepsilon\,F - C_\varepsilon R^{-1}\;\;\text{in}\;\; B_{\left( \lambda- \delta\right) \vartheta R} \times \left( 0, \tfrac{1}{2n} R^2 \right)\,,
\end{equation}
where $C_\varepsilon= c\left(n,f,\alpha,q,\varepsilon \right) \Theta \left( \frac{\Theta\, V}{1-\vartheta} \right)^{\frac{2}{q}}$ and $c_m^{-1} := f(\underbrace{0,\dots,0}_{m\text{-times}},1,\dots,1)$.
\end{theorem}

We refer to the estimates \eqref{eq:cylindrical convex speeds} as ``cylindrical estimates'' since (under the conditions imposed on $f$) a weakly locally convex hypersurface with $F>0$ satisfying
\[
\kappa_1+\dots+\kappa_{m+1}\ge c_mF
\]
must either satisfy
\[
\kappa_1+\dots+\kappa_m>0\;\;\text{or else}\;\; \kappa_1=\dots=\kappa_m=0\;\;\text{and}\;\;\kappa_{m+1}=\dots=\kappa_n>0
\]
Note that the $m=0$ case provides a local ``umbilic estimate''.

It suffices to prove the estimates in case $R=1$, which is exactly what we shall do. This frees up the symbol $R$ to be used as in the definition of the cut-off functions above. We hope that this shall cause no confusion.

\subsection{Proof of the convexity estimate for flows of hypersurfaces by convex admissible speeds}\label{sec:convexity estimate for convex speeds}

\newtheorem{parabolic}{Lemma}[section]
\newtheorem{bounds}[parabolic]{Lemma}
\newtheorem{evlnGes}[parabolic]{Lemma}
\newtheorem{Stampacchia}[parabolic]{Lemma}




Given a solution $X:M^n\times[0,T)\to\mathbb{R}^{n+1}$ to such a flow, consider the curvature pinching function $G:M^n\times[0,T)\to\mathbb{R}$ introduced in \cite[Section 3]{MR3218814}. This function is given by a smooth, non-negative, symmetric, one-homogeneous function of the principal curvatures, vanishing if and only if $\kappa_1\ge 0$. Consider also a cut-off function $\psi=\zeta\circ X$ as prescribed above.

Given $\varepsilon>0$, $\sigma\in(0,1)$ and $p>10$, define
\[
G_{\varepsilon,\sigma}:=(G-\varepsilon F)F^{\sigma-1}\;\;\text{and}\;\; v :=(G_{\varepsilon,\sigma})_+^{\frac{p}{2}}\,.
\]
Our goal is to establish a good $L^2$ estimate for $\psi v$ for some (sufficiently large) $p$, which we will then be able to bootstrap to an $L^\infty$ estimate by Stampacchia iteration. To that end, we establish the following inequality.
\begin{lemma}
If $p\ge L=L(n,f,\Gamma^n_0,\varepsilon)$, then
\begin{equation}\label{eq:evolve psi v}
\frac{\left(\partial_t-\mathcal{L} \right) \psi^2 v^2}{\psi^2 v^2}\le \frac{150n\Lambda\chi_{B_R\setminus B_r}}{(R-r)^2\psi}+\sigma p\Lambda\vert A\vert^2-2\gamma p\frac{\vert \nabla A\vert^2}{F^2}-3\Lambda^{-1}\frac{\vert\nabla v\vert^2}{v^2}
\end{equation}
at points where $\psi v>0$, for some $C=C(n,f,\Gamma^n_0,\varepsilon)<\infty$ and $\gamma=\gamma(n,f,\Gamma^n_0,\varepsilon)>0$.
\end{lemma}
\begin{proof}
We first recall (as in \cite[Lemma 3.4]{MR3218814}, notwithstanding the typo there in the second last term) that
\begin{equation}\label{eq:evlnGes}
\begin{split}
\left( \partial_t - \mathcal{L} \right) G_{\varepsilon, \sigma} ={}& F^{\sigma-1} \left(\dot G^{kl} \ddot F^{pq, rs} - \dot F^{kl} \ddot G^{pq, rs} \right) \nabla_k A_{pq} \nabla_l A_{rs}\\
{}&  + 2\left( 1-\sigma \right) \left< \nabla G_{\varepsilon, \sigma}, \frac{\nabla F}{F} \right>_F - \sigma(1 - \sigma)G_{\varepsilon, \sigma}\frac{\left| \nabla F\right|_F^2}{F^2} + \sigma G_{\varepsilon, \sigma} \left| A \right|^2_F \mbox{.}
\end{split}
\end{equation}
In fact, this equation holds when $G$ is given by \emph{any} smooth, symmetric, one-homogeneous function of the principal curvatures. A key point of the construction of the particular $G$ is to ensure that the first term on the right of \eqref{eq:evlnGes} is useful (the negative penultimate term will simply be discarded as its coefficient is too small to be of use). In the case at hand, \cite[Lemma 3.3]{MR3218814} guarantees that there is a constant $\gamma=\gamma(n,f,\Gamma^n_0,\varepsilon)>0$ such that, wherever\footnote{This is the main point at which the analysis of flows of surfaces 
by general speeds differs from flows by convex speeds.} $G_{\varepsilon,\sigma}>0$,
\begin{equation}\label{eq:Q estimate convex case}
F^{\sigma-1} \left(\dot G^{kl} \ddot F^{pq, rs} - \dot F^{kl} \ddot G^{pq, rs} \right) \nabla_k A_{pq} \nabla_l A_{rs}\le -\gamma G_{\varepsilon,\sigma}\frac{\vert\nabla A\vert^2}{F^2} \mbox{.}
\end{equation}
Discarding the penultimate term of \eqref{eq:evlnGes} and applying Young's inequality to the gradient term, we may thus estimate
\begin{equation}\label{eq:key estimate for ev Ges}
(\partial_t-\mathcal{L})G_{\varepsilon,\sigma}\le G_{\varepsilon,\sigma}\left(\sigma\vert A\vert^2_F-\gamma\frac{\vert\nabla A\vert^2}{F^2}+\gamma^{-1}\frac{\vert\nabla G_{\varepsilon,\sigma}\vert^2}{G_{\varepsilon,\sigma}^2}\right)
\end{equation}
wherever $G_{\varepsilon,\sigma}>0$ for some (potentially smaller than before) $\gamma=\gamma(n,f,\Gamma^n_0,\varepsilon)>0$.

Given that (at points where $\psi v>0$)
\begin{equation*}
\frac{\left( \partial_t - \mathcal{L} \right) \psi^2 v^2}{\psi^2 v^2} = 2 \frac{\left( \partial_t - \mathcal{L} \right) \psi}{\psi} - 2 \frac{\left| \nabla \psi\right|_F^2}{\psi^2} + \frac{p\left( \partial_t - \mathcal{L}\right) G_{\varepsilon, \sigma}}{G_{\varepsilon, \sigma}} - 4\frac{p-1}{p} \frac{\left| \nabla v\right|_F^2}{v^2} - 8 \left< \frac{\nabla \psi}{\psi}, \frac{\nabla v}{v}\right>_F\,,
\end{equation*}
the claim follows fairly straightforwardly from the key properties of the cut-off function and careful absorption using Young's inequality and \eqref{eq:uniform parabolic} (cf. \cite{MR3218814} and \cite{MR4609824}).
\end{proof}
Recalling that (see, e.g., \cite[Lemma 2.3]{MR3218814})
\[
\frac{d}{dt}d\mu=-FHd\mu\,,
\]
consider now
\begin{align*}
\frac{d}{dt}\int\psi^2v^2\,d\mu={}&\int\left(\frac{(\partial_t-\mathcal{L})\psi^2v^2}{\psi^2v^2}-FH\right)\psi^2v^2\,d\mu+\int \mathcal{L}(\psi^2v^2)\,d\mu\,.
\end{align*}
In contrast to the mean curvature flow, the last integral on the right hand side may not vanish for flows by a general admissible speed. Nonetheless, the divergence theorem yields (by an abuse of notation)
\begin{multline*}
\int\mathcal{L}(\psi^2v^2)\,d\mu={}\int\dot F^{ij}\nabla_i\nabla_j(\psi^2v^2)\,d\mu\\
={}-\int\ddot F^{ij,kl}\nabla_iA_{lk}\nabla_j(\psi^2v^2)\,d\mu
\le{}C\int\frac{\vert\nabla A\vert}{F}\left\vert\frac{\nabla\psi}{\psi}+\frac{\nabla v}{v}\right\vert\psi^2v^2\,d\mu\,,
\end{multline*}
where $C=C(n,f,\Gamma^n_0)$, since homogeneity of $F$ implies that $\ddot F$ is comparable to $F^{-1}$ in $\Gamma^n_0$ (note that convexity of $F$ is not used here). These terms are easily absorbed for large $p$ using Young's inequality, yielding
\begin{align*}
\frac{d}{dt}\int\psi^2v^2\,d\mu{}&+\Lambda^{-1}\int H^2\psi^2v^2\,d\mu\\
\le{}&\frac{200\, n\, \Lambda}{(R-r)^2}\int_{B_R\setminus B_r} v^2\,d\mu+\int\left(\sigma p\Lambda\vert A\vert^2-2\Lambda^{-1}\frac{\vert\nabla v\vert^2}{v^2}-\gamma p\frac{\vert\nabla A\vert^2}{F^2}\right)\psi^2v^2\,d\mu\,.
\end{align*}
Applying the Poincar\'e-type inequality (Theorem \ref{prop:Poincare}) with $u=\psi \, v$ (cf. \cite{MR4609824}) now yields
\begin{align}\label{eq:d/dt L2}
\frac{d}{dt}\int\psi^2v^2\,d\mu+\Lambda^{-1}\int\left(\vert\nabla v\vert^2+ v^2H^2\right)\psi^2\,d\mu\le{}&\frac{200\,n\, \Lambda}{(R-r)^2}\int_{B_R\setminus B_r}v^2\,d\mu
\end{align}
for some $C=C(n,f,\Gamma^n_0)<\infty$, so long as $p\ge L=L(n,f,\Gamma^n_0,\varepsilon)$ and $\sigma \, p^{\frac{1}{2}}\le \ell=\ell(n,f,\Gamma^n_0,\varepsilon)$.

If we choose $R=\lambda$ and $r=\lambda-\delta$, then discarding the two good terms on the left and integrating with respect to time yields
\[
\sup_{t\in (0,\frac{1}{2n})}\int_{B_{\lambda-\delta}}v^2\,d\mu\le\int_{B_\lambda}v^2\,d\mu_0+200\, n\, \Lambda\delta^{-2}\int\!\!\!\int_{B_\lambda\setminus B_{\lambda-\delta}}v^2\,d\mu\,dt\,.
\]
Integrating again with respect to $t$ yields the following $L^2$ estimate.
\begin{lemma}[$L^2$-estimate]\label{lem:L2 convex speeds}
If $p\ge L=L(n,f,\Gamma^n_0,\varepsilon)$ and $\sigma \, p^{\frac{1}{2}}\le \ell=\ell(n,f,\Gamma^n_0,\varepsilon)$, then
\begin{equation}\label{eq:L2 estimate}
\int\!\!\!\int_{B_{\lambda-\delta}}v^2\,d\mu\,dt\le\frac{1}{2n}\int_{B_\lambda}v^2\,d\mu_0+100\Lambda\delta^{-2}\int\!\!\!\int_{B_\lambda\setminus B_{\lambda-\delta}}v^2\,d\mu\,dt\,.
\end{equation}
\end{lemma}

Now we apply Stampacchia's lemma (very similarly as in \cite{MR4609824}) to convert \eqref{eq:L2 estimate} into an $L^\infty$ estimate. To this end we consider, for each $k> k_0 := \Theta^\sigma \geq \sup_{B_\lambda \backslash B_{\lambda - \delta}} G_{\varepsilon, \sigma}$ and $R< \lambda - \delta$,
\[
v_k^2 := \left( G_{\varepsilon, \sigma}- k \right)_+^p \;\mbox{ and }\; A_{k, R}\left( t\right) := \left\{ \left( x, t\right) \in X^{-1}\left( B_R\right) : v_k\left( x, t\right) > 0 \right\}\,.
\]
Set also
\[
u\left( k, R\right) := \int\!\!\!\int_{A_{k, R}} v_k^2 d\mu\, dt \;\mbox{ and }\; a\left( k, R\right) := \int\!\!\!\int_{A_k, R} d\mu\, dt \mbox{.}
\]
Observe that, for any $h\geq k >0$ and any $0<r\leq R \leq \lambda -\delta$,
\[
\left( h-k\right)^p a\left( h, r\right) \leq u\left( k, r\right) \mbox{.}
\]
In order to estimate $u\left( k, r\right)$, we observe first that the same arguments which led to \eqref{eq:d/dt L2} yield  (cf. \cite{MR4609824})
\begin{multline}\label{E:vk}
\frac{d}{dt} \int \psi^2 v_k^2 d\mu + \Lambda^{-1}\int\left(\vert\nabla v_k\vert^2+v_k^2H^2\right)\psi^2\,d\mu\\
\leq \frac{C}{\left( R-r\right)^2} \int_{A_{k, R}} v_k^2 d\mu + \sigma p \Lambda\int_{A_{k, R}} G_{\varepsilon, \sigma}^p \left| A \right|^2 d\mu
\end{multline}
for some $C=C(n,f,\Gamma^n_0)<\infty$, so long as $p\ge L=L(n,f,\Gamma^n_0,\varepsilon)$ and $\sigma p^{-\frac{1}{2}}\ge \ell=\ell(n,f,\Gamma^n_0,\varepsilon)$. 
Applying \eqref{eq:MichaelSimon} when $n\ge 3$ (see the alternative argument at the conclusion of this proof for the $n=2$ case) with $u=\psi\, v_k$ yields
\begin{equation*} 
\frac{d}{dt} \int \psi^2 v_k^2 d\mu + C^{-1}\left( \int_{A_{k, r}}  v_k^{2^*} d\mu \right)^{\frac{2}{2*}} 
\leq \frac{C}{\left(R-r\right)^2} \int_{A_{k, R}} v_k^2 d\mu + \sigma p \Lambda \int_{A_{k, R}} G_{\varepsilon, \sigma}^p \left| A \right|^2 d\mu\,,
\end{equation*}
where $\frac{1}{2^\ast}:=\frac{1}{2}-\frac{1}{n}$ and $C=C(n,f,\Gamma^n_0)<\infty$.

Integrating with respect to $t$ then yields
\begin{multline*}
\sup_{t\in \left( 0, \frac{1}{2n} R^2 \right)} \int_{A_{k, r}} v_k^2 d\mu + \int \left(  \int_{A_{k, r}}  v_k^{2^*} d\mu \right)^{\frac{2}{2^*}} dt\\
\leq C\left(\frac{1}{\left( R - r\right)^2} \int\!\!\!\int_{A_{k, R}} v_k^2 d\mu \, dt + \sigma p\int_{A_{k, R}} G_{\varepsilon, \sigma}^p \left| A \right|^2 d\mu \, dt\right) \mbox{,}
\end{multline*}
where $C=C(n,f,\Gamma^n_0)$. Estimating $\vert A\vert^2\le C(n,f,\Gamma^n_0)F^2$ and using interpolation and Young's inequalities exactly as in \cite{MR4609824}, we deduce that
\[
\left(\int\!\!\!\int_{A_{k, r}} v_k^{2q^\ast} d\mu\, dt \right)^{\frac{1}{q^\ast}} \leq C\left(\frac{1}{\left( R - r\right)^2}\int\!\!\!\int_{A_{k, R}} v_k^2 d\mu \, dt + \sigma p\int_{A_{k, R}} G_{\varepsilon, \sigma}^pF^2 d\mu \, dt\right)\mbox{,}
\]
where $q_\ast:=2-\frac{2}{2^\ast}$, which is just $\frac{n+2}{n}$ when $n\ge 3$. (We shall see below that the exponent in the corresponding estimate can be made arbitrarily close to this number when $n=2$). Next note that $G_{\varepsilon,\sigma}^pF^2=G_{\varepsilon,\sigma'}^p$ where $\sigma':=\sigma + \frac{2}{p}$. So H\"older's inequality gives
\[
\int\!\!\!\int_{A_{k, R}}G_{\varepsilon, \sigma}^pF^2\,d\mu\, dt \leq a\left( k, R \right)^{1-\frac{1}{\rho}} \left(\int\!\!\!\int_{B_{\lambda-\delta}} G_{\varepsilon, \sigma', +}^{p \, \rho} d\mu \, dt \right)^{\frac{1}{\rho}}
\]
for any $\rho\geq 1$ (to be determined). Note that we may choose $\ell$ smaller, depending now also on $\rho$, so that \eqref{eq:L2 estimate} applies with $\sigma$ replaced by $\sigma'$ and $p$ replaced by $\rho \, p$.

Similarly, we may estimate, for any $k>k_0$ and $R< \lambda -\delta$,
\[
u\left( k, R\right) \leq a\left( k, R \right)^{1- \frac{1}{\rho}} \left(\int\!\!\!\int_{B_{\lambda-\delta}} G_{\varepsilon, \sigma', +}^{p \, \rho} d\mu \, dt \right)^{\frac{1}{\rho}} \mbox{.}
\]
Finally, we estimate
\[
u\left( k, r\right) \leq a\left( k, r\right)^{\frac{2}{n+2}} \left(\int\!\!\!\int_{A_{k, r}} v_k^{\frac{2\left( n+2\right)}{n}} d\mu\, dt \right)^{\frac{n}{n+2}} \mbox{.}
\]
Using these and the $L^2$ estimate \eqref{eq:L2 estimate} we obtain similarly as in \cite{MR4609824},
\[
\left( h-k\right)^p \left( R-r\right)^2 a\left( h, r\right) \leq C\, a\left( k, R\right)^{\gamma} \left[ 1 + \sigma p\, \Theta^2\left( \lambda - \delta \right)^2 \right] \Theta^{\sigma\, p \left( 1 - \frac{1}{\rho}\right) }\Omega^{\frac{\sigma p}{\rho}} \mbox{,}
\]
provided $p>L\left( n, \alpha, \varepsilon, \rho \right)$ and $\sigma \leq \ell\left( n, \alpha, \varepsilon, \rho\right) p^{-\frac{1}{2}}$, where $C=C\left(n,f,\Gamma^n_0,\rho\right)$, $\gamma := 1 + \frac{2}{n+2} - \frac{1}{\rho}$ if $n\ge 3$ (or $\gamma := 1 + \frac{2}{n+2} - \frac{1}{\rho}+\frac{2}{2^\ast}$ if $n=2$), and
\[
\Omega := \left(\frac{1}{2n}\int F^{\sigma p} d\mu_0 + \frac{100\Lambda}{\delta^2} \int \int_{B_\lambda \backslash B_{\lambda-\delta}}F^{\sigma p} d\mu\, dt \right)^{\frac{1}{\sigma p}} \mbox{.}
\]

We now fix $\rho=\rho\left( n \right) > 1 + \frac{n}{2}$, which ensures $\gamma = \gamma\left( n \right) >1$, and choose $p=p\left( n, \alpha, \varepsilon, q \right) \geq L$ and $\sigma = \sigma\left( n, \alpha, \varepsilon, q \right) \leq \ell \, p^{-\frac {1}{2}}$ such that $\sigma p = q$.  Lemma \ref{T:Stampacchia} gives
\[
a\left( k_0 + d, \vartheta R_0 \right) = 0 \mbox{,}
\]
where $R_0:= \lambda - \delta$ and 
\[
d^p:= \frac{2^{\frac{\left( p+2\right) \gamma}{\gamma-1}} c\left[ 1 + \sigma p\, \Theta^2\left( \lambda - \delta \right)^2 \right] \Theta^{\sigma p \left( 1 - \frac{1}{\rho}\right)}\Omega^{\sigma \, p}{\rho} a\left( k_0, R_0\right)^{\gamma-1}}{\left( 1 - \vartheta\right)^2 \left(\lambda-\delta\right)^2} \mbox{.}
\]
Using \eqref{eq:L2 estimate} we may estimate, as in \cite{MR4609824},
\[
a\left( k_0, R_0\right) \leq k_0^{-p} \Omega^{\sigma p}\,.
\]
Since $k_0= \Theta^\sigma$ and $\sigma p = q$ we obtain, as in \cite{MR4609824}, 
\[
G\leq 2\varepsilon F + c\left( n, \alpha, q, \varepsilon \right) \Theta \left( \frac{\Theta \, V}{1- \vartheta} \right)^{\frac{2}{q}}
\]
in $B_{\vartheta\left( \lambda - \delta \right)}$.  This completes the proof of Theorem \ref{thm:convexityc} in the case $n\ge 3$.

In the case $n=2$, we apply the Poincar\'e inequality \ref{T:Poincare} (instead of the Sobolev inequality) to \eqref{E:vk} to obtain
\begin{equation*} 
\frac{d}{dt} \int \psi^2 v_k^2 d\mu + \frac{1}{Ca(k,R)^{\frac{1}{s}}}\left(\int_{A_{k, r}}v_k^{2s} d\mu \right)^{\frac{1}{s}} 
\leq \frac{C}{\left(R-r\right)^2} \int_{A_{k, R}} v_k^2 d\mu + \sigma p \Lambda \int_{A_{k, R}} G_{\varepsilon, \sigma}^p \left| A \right|^2 d\mu
\end{equation*}
for any choice of $s>1$ (say $s=2$ for concreteness), where $C=C(n,f,\Gamma^n_0)$. But since the $L^2$ estimate and the hypotheses guarantee that
\[
a(k,R)\le\frac{u(0,R)}{k^p}\le \frac{\Omega^{\sigma p}}{k^p}=\frac{\Omega^{q}}{k^p}\le \frac{100\Lambda\Theta^qV^{n+2}}{k^p}\le 1
\]
for $k\ge k(n,f,\Gamma^n_0,q,\Theta,V)$ sufficiently large, we may now proceed as for $n\ge 3$ to complete the proof in the case $n=2$.
\hspace*{\fill}$\Box$

\subsection{Proof of the cylindrical (and umbilic) estimates for flows of hypersurfaces by convex admissible speeds}

We next consider, for any $m=\{0,\dots,n-2\}$, the flow by any convex admissible speed function $f:\Gamma^2\to\mathbb{R}$ with $\Gamma^n_{m+1}\subset\Gamma^n$. 
Along any solution $X:M^n\times[0,T)\to\mathbb{R}^{n+1}$ to this flow, we define the pinching function $G:M^n\times[0,T)\to\mathbb{R}^{n+1}$ as in \cite[Section 3]{MR3265960}. This function is given by a smooth, one-homogeneous, symmetric function of the principal curvatures which is non-negative in $\Gamma^n_{m+1}$ and vanishes precisely when $\kappa_1+\dots+\kappa_{m+1}\ge c_mF$, where $c_m$ is as defined in the statement of Theorem \ref{thm:cylindrical convex}.
Moreover, if we set $G_{\varepsilon,\sigma}:=(G-\varepsilon F)F^{\sigma-1}$, then due to \cite[Proposition 3.4]{MR3265960}, we may estimate
\[
(\partial_t-\mathcal{L})G_{\varepsilon,\sigma}\le G_{\varepsilon,\sigma}\left(\sigma\vert A\vert^2_F-\gamma\frac{\vert\nabla A\vert^2}{F^2}+\gamma^{-1}\frac{\vert\nabla G_{\varepsilon,\sigma}\vert^2}{G_{\varepsilon,\sigma}^2}\right)
\]
for some $\gamma=\gamma(n,f,\alpha,\varepsilon)>0$ wherever $G_{\varepsilon,\sigma}>0$. 
We may thus proceed exactly as in Section \ref{sec:convexity estimate for convex speeds} (from \eqref{eq:key estimate for ev Ges} onwards) to establish Theorem \ref{thm:cylindrical convex}.

\section{Surfaces evolving by general speeds} \label{S:surfaces}

We next consider the deformation of surfaces by general admissible speed functions $f:\Gamma^n\to\mathbb{R}$. Such flows preserve any convex cone of curvatures, and hence remain uniformly parabolic throughout the evolution \cite[Corollary 2.3]{MR3299822} (cf. \cite{MR2729317}). 



We will establish the following local umbilic and convexity estimates for general admissible flows of surfaces.

\begin{theorem}[Umbilic estimate --- surfaces]\label{thm:umbilic surfaces}
Let $f:\Gamma^2_+\to\R$ be an admissible speed function. If a solution to the flow with speed $F=f(\kappa_1,\kappa_2)$ is properly defined in the spacetime cylinder $\overline B_{\lambda R} \times \left[ 0, \frac{1}{4} R^2 \right) \subset \mathbb{R}^{3} \times \mathbb{R}$ and satisfies
\begin{enumerate}
\item $\displaystyle\inf_{B_{\lambda R} \times \left\{ 0 \right\} \cup \partial B_{\lambda R} \times \left( 0, \frac{1}{4} R^2 \right)} \frac{\kappa_1}{\kappa_2} \geq \alpha > 0$,
\item $\displaystyle\sup_{B_{\lambda R} \times \left\{ 0 \right\} \cup B_{\lambda R}\setminus B_{(\lambda-\delta)R} \times \left( 0, \frac{1}{4} R^2 \right)} F \leq \Theta\, R^{-1}$, and
\item $\displaystyle\frac{R^2}{4} \int_{B_{\lambda R}} F^q d\mu_0 + \frac{1}{\delta^2} \int_0^{\frac{R^2}{4}} \int_{B_{\lambda R} \backslash B_{(\lambda-\delta)R}} F^q d\mu\, dt \leq \left( \Theta\, R^{-1}\right)^q \left( V \, R\right)^{4}$,
\end{enumerate}
then, given any $\varepsilon>0$ and $\vartheta \in \left(0,1\right)$, it satisfies
\begin{equation} \label{eq:umbilic surfaces}
\kappa_2-\kappa_1\le \varepsilon\,F + C_\varepsilon R^{-1}\;\;\text{in}\;\; B_{\left( \lambda- \delta\right) \vartheta R} \times \left( 0, \tfrac{1}{4} R^2 \right)\,,
\end{equation}
where $C_\varepsilon= c\left(f,\alpha,q,\varepsilon \right) \Theta \left( \frac{\Theta\, V}{1-\vartheta} \right)^{\frac{2}{q}}$.
\end{theorem}

\begin{theorem}[Convexity estimate --- surfaces]\label{thm:convexity surfaces}
Let $f:\Gamma^2\to\R$ be an admissible speed function. If a solution to the flow with speed $F=f(\kappa_1,\kappa_2)$ is properly defined in the spacetime cylinder $\overline B_{\lambda R} \times \left[ 0, \frac{1}{4} R^2 \right) \subset \mathbb{R}^{3} \times \mathbb{R}$ and satisfies
\begin{enumerate}
\item $\displaystyle\inf_{B_{\lambda R} \times \left\{ 0 \right\} \cup \partial B_{\lambda R} \times \left( 0, \frac{1}{4} R^2 \right)} \frac{\kappa_1}{\kappa_2} \geq -\alpha^{-1}$, $\alpha > 0$,
\item $\displaystyle\sup_{B_{\lambda R} \times \left\{ 0 \right\} \cup B_{\lambda R}\setminus B_{(\lambda-\delta)R} \times \left( 0, \frac{1}{4} R^2 \right)} F \leq \Theta\, R^{-1}$, and
\item $\displaystyle\frac{R^2}{4} \int_{B_{\lambda R}} F^q d\mu_0 + \frac{1}{\delta^2} \int_0^{\frac{R^2}{4}} \int_{B_{\lambda R} \backslash B_{(\lambda-\delta)R}} F^q d\mu\, dt \leq \left( \Theta\, R^{-1}\right)^q \left( V \, R\right)^{4}$,
\end{enumerate}
then, given any $\varepsilon>0$ and $\vartheta \in \left(0,1\right)$, it satisfies
\begin{equation} \label{eq:convexity surfaces}
\kappa_1 \geq -\varepsilon\,F - C_\varepsilon R^{-1}\;\;\text{in}\;\; B_{\left( \lambda- \delta\right) \vartheta R} \times \left( 0, \tfrac{1}{4} R^2 \right)\,,
\end{equation}
where $C_\varepsilon= c\left(f,\alpha,q,\varepsilon \right) \Theta \left( \frac{\Theta\, V}{1-\vartheta} \right)^{\frac{2}{q}}$.
\end{theorem}

\newtheorem{evlnGs}{Lemma}[section]

\subsection{Proof of the umbilic estimate for flows of surfaces by general admissible speeds}

Following \cite[Proposition 19.5]{MR4249616}, consider the function $G:=1-\frac{\kappa_1}{\kappa_2}$.

\begin{lemma}\label{lem:evolve G surfaces}
There exists $\gamma=\gamma(f,\alpha,\varepsilon)>0$ such that 
\[
\frac{(\partial_t-\mathcal{L})G}{G}\le -\gamma\frac{\vert\nabla A\vert^2}{F^2}+\gamma^{-1}\frac{\vert \nabla G\vert^2}{G^2}
\]
wherever $\alpha\le\frac{\kappa_1}{\kappa_2}\le1-\varepsilon$.
\end{lemma}
\begin{proof}
Away from points of principal curvature multiplicity (umbilic points in the case of surfaces), an (evolving) hypersurface admits a smooth frame of principal directions, $\{e_i\}_{i=1}^n$. In particular, the principal curvatures $\kappa_i$ are smooth at such points. Differentiating the eigenvalue equations $A(e_i)=\kappa_ie_i$ then yields the identities
\[
\nabla_k\kappa_i=\nabla_kA_{ii}\,,\;\;
\nabla_k\nabla_\ell\kappa_i=\nabla_k\nabla_\ell A_{ii}-2\sum_{p\ne i}\frac{\nabla_kA_{ip}\nabla_\ell A_{ip}}{\kappa_i-\kappa_p}\,,\;\;\text{and}\;\;
\partial_t\kappa_i=\nabla_tA_{ii}\,.
\]
Since (away from points of principal curvature multiplicity) the second fundamental form evolves  according to
\[
(\nabla_t-\mathcal{L})A_{ij}=\vert A\vert_F^2+\ddot F^{pq}\nabla_iA_{pp}\nabla_jA_{qq}+2\sum_{p>q}\frac{\dot F^p-\dot F^q}{\kappa_p-\kappa_q}\nabla_iA_{pq}\nabla_jA_{pq}\,,
\]
we find, in the case at hand, that 
\[
(\partial_t-\mathcal{L})\kappa_1
=\vert A\vert^2_F\kappa_1+\ddot F^{pq}\nabla_1A_{pp}\nabla_1A_{qq}+2\frac{\dot F^2}{\kappa_2-\kappa_1}\left[(\nabla_1A_{12})^2+(\nabla_2A_{12})^2\right]
\]
and
\[
(\partial_t-\mathcal{L})\kappa_2
=\vert A\vert^2_F\kappa_2+\ddot F^{pq}\nabla_2A_{pp}\nabla_2A_{qq}-2\frac{\dot F^1}{\kappa_2-\kappa_1}\left[(\nabla_1A_{12})^2+(\nabla_2A_{12})^2\right]
\]
away from umbilic points. Thus, wherever $\alpha\le\frac{\kappa_1}{\kappa_2}\le1-\varepsilon$, the function $G_0:=\frac{\kappa_1}{\kappa_2}$ is smooth and satisfies
\begin{equation}\label{eq:surfaces G grad identity}
\frac{\nabla_kG_0}{G_0}=\frac{\nabla_kA_{11}}{\kappa_1}-\frac{\nabla_kA_{22}}{\kappa_2}
\end{equation}
and
\begin{equation} \label{E:G0}
\begin{split}
\frac{(\partial_t-\mathcal{L})G_0}{G_0}
={}&\ddot F^{pq}\left(\frac{\nabla_1A_{pp}\nabla_1A_{qq}}{\kappa_1}-\frac{\nabla_2A_{pp}\nabla_2A_{qq}}{\kappa_2}\right)\\{}&+\frac{2F}{\kappa_1\kappa_2(\kappa_2-\kappa_1)}\left[(\nabla_1A_{12})^2+(\nabla_2A_{12})^2\right]+2\left\langle\frac{\nabla G_0}{G_0},\frac{\nabla A_{22}}{\kappa_2}\right\rangle_F\,,
\end{split}
\end{equation}
where we used the identity
\[
\dot F^{1}\kappa_1+\dot F^{2}\kappa_2=F\,,
\]
which is an immediate consequence of Euler's theorem for homogeneous functions. Using \eqref{eq:surfaces G grad identity} to replace $\nabla_1A_{11}$ and $\nabla_2A_{22}$, and the identity
\[
\ddot F^{11}\kappa_1^2+2\ddot F^{12}\kappa_1\kappa_2+\ddot F^{22}\kappa_2^2=0\,,
\]
which also follows from Euler's theorem for homogeneous functions, we may rewrite the terms on the first line of \eqref{E:G0} as (cf. \cite[Proposition 19.5]{MR4249616})
\begin{align*}
\ddot F^{pq}\left(\frac{\nabla_1A_{pp}\nabla_1A_{qq}}{\kappa_1}\right.{}&-\left.\frac{\nabla_2A_{pp}\nabla_2A_{qq}}{\kappa_2}\right)\\
=\ddot F^{11}\kappa_1&\left(2\frac{\nabla_1G_0}{G_0}\frac{\nabla_1A_{22}}{\kappa_2}+\frac{(\nabla_1G_0)^2}{G_0^2}\right)+2\ddot F^{12}\left(\frac{\nabla_1G_0}{G_0}\nabla_1A_{22}+\frac{\nabla_2G_0}{G_0}\nabla_2A_{11}\right)\\
+\ddot F^{22}&\kappa_2\left(2\frac{\nabla_2G_0}{G_0}\frac{\nabla_2A_{11}}{\kappa_2}-\frac{(\nabla_2G_0)^2}{G_0^2}\right)\,.
\end{align*}
Since, due to the homogeneity of $f$, $|\ddot F^{pq}|\le C(\alpha,f)F^{-1}$, we may thereby estimate
\[
\ddot F^{pq}\left(\frac{\nabla_1A_{pp}\nabla_1A_{qq}}{\kappa_1}-\frac{\nabla_2A_{pp}\nabla_2A_{qq}}{\kappa_2}\right)\ge -C\left(\frac{\vert\nabla G_0\vert^2}{G_0^2}+\frac{\vert\nabla A\vert}{F}\frac{\vert\nabla G_0\vert}{G_0}\right)
\]
for some $C=C(f,\alpha)<\infty$.

We next use \eqref{eq:surfaces G grad identity} to write
\begin{align*}
(\nabla_1A_{22})^2+{}&(\nabla_2A_{11})^2=(\nabla_1A_{22})^2+a^2(\nabla_1A_{11})^2-a^2\kappa_1^2\left(\frac{\nabla_1A_{22}}{\kappa_2}+\frac{\nabla_1G_0}{G_0}\right)^2\\
{}&\qquad\qquad\quad+(\nabla_2A_{11})^2+b^2(\nabla_2A_{22})^2-b^2\kappa_2^2\left(\frac{\nabla_2A_{11}}{\kappa_1}-\frac{\nabla_2G_0}{G_0}\right)^2\\
={}&\left(1-a^2G_0^2\right)(\nabla_1A_{22})^2+a^2(\nabla_1A_{11})^2+\left(1-b^2G_0^{-2}\right)(\nabla_2A_{11})^2+b^2(\nabla_2A_{22})^2\\
{}&-a^2\kappa_1^2\left(\frac{(\nabla_1G_0)^2}{G_0^2}+2\frac{\nabla_1A_{22}}{\kappa_2}\frac{\nabla_1G_0}{G_0}\right)-b^2\kappa_2^2\left(\frac{(\nabla_2G_0)^2}{G_0^2}-2\frac{\nabla_2A_{11}}{\kappa_1}\frac{\nabla_2G_0}{G_0}\right)\,.
\end{align*}
Since, by hypothesis, $\alpha\le G_0\le 1-\varepsilon$, taking $a=\frac{1}{1+G_0^2}$ and $b=\frac{G_0^2}{1+G_0^2}$, we obtain
\[
(\nabla_1A_{22})^2+(\nabla_2A_{11})^2\ge \gamma\vert\nabla A\vert^2-C\left(\frac{\vert\nabla G_0\vert^2}{G_0^2}+\frac{\vert\nabla G_0\vert}{G_0}\frac{\vert\nabla A\vert}{F}\right)
\]
for some $\gamma=\gamma(f,\alpha,\varepsilon)>0$ and $C=C(f,\alpha,\varepsilon)<\infty$. The claim now follows by absorbing remaining terms using Young's inequality. 
\end{proof}

In particular, we find that the inequality $\kappa_1\ge\alpha \, \kappa_2$ is preserved, and hence
\[
\Lambda^{-1}\delta^{ij}\le \dot F^{ij}\le\Lambda\, \delta^{ij}
\]
for some $\Lambda=\Lambda(f,\alpha)$.

Given $\varepsilon>0$ and $\sigma\in(0,1)$, consider the function $G_{\varepsilon,\sigma}:=(G-\varepsilon)F^\sigma$. Since $G_\varepsilon:=G-\varepsilon\le G\le 1$, we find, wherever $G_{\varepsilon,\sigma}>0$, that
\begin{multline*}
(\partial_t-\mathcal{L})G_{\varepsilon,\sigma}\\
\le{} G_{\varepsilon,\sigma}\left(\sigma\vert A\vert^2_F+\gamma^{-1}\frac{\vert\nabla G_\varepsilon\vert^2}{G_\varepsilon^2}-\gamma \frac{\vert\nabla A\vert^2}{F^2}+\sigma(1-\sigma)\frac{\vert\nabla F\vert^2_F}{F^2}-2\sigma \left\langle\frac{\nabla G_\varepsilon}{G_\varepsilon},\frac{\nabla F}{F}\right\rangle_F\right).
\end{multline*}
Replacing incidences of $\nabla G_\varepsilon$ using the identity
\[
\frac{\nabla G_{\varepsilon,\sigma}}{G_{\varepsilon,\sigma}}=\frac{\nabla G_\varepsilon}{G_\varepsilon}-\sigma\frac{\nabla F}{F}
\]
and applying Young's inequality then yields
\begin{align*}
\frac{(\partial_t-\mathcal{L})G_{\varepsilon,\sigma}}{G_{\varepsilon,\sigma}}\le{}&\sigma\vert A\vert^2_F-\gamma \frac{\vert\nabla A\vert^2}{F^2}+\gamma^{-1}\left(\frac{\vert\nabla G_{\varepsilon,\sigma}\vert^2}{G_{\varepsilon,\sigma}^2}+\sigma\frac{\vert\nabla F\vert^2}{F^2}\right)
\end{align*}
wherever $G_{\varepsilon,\sigma}>0$ for some (much smaller than before) $\gamma=\gamma(f,\alpha,\varepsilon)$. We conclude that
\begin{align*}
\frac{(\partial_t-\mathcal{L})G_{\varepsilon,\sigma}}{G_{\varepsilon,\sigma}}\le{}&\sigma\vert A\vert^2_F-\gamma \frac{\vert\nabla A\vert^2}{F^2}+\gamma^{-1}\frac{\vert\nabla G_{\varepsilon,\sigma}\vert^2}{G_{\varepsilon,\sigma}^2}
\end{align*}
for some (smaller than before) $\gamma=\gamma(f,\alpha,\varepsilon)>0$, so long as $\sigma\le\ell=\ell(f,\alpha,\varepsilon)$.

We may thus proceed exactly as in Section \ref{sec:convexity estimate for convex speeds} (from \eqref{eq:key estimate for ev Ges} onwards) to establish Theorem \ref{thm:umbilic surfaces}.

\subsection{Proof of the convexity estimate for flows of surfaces by general admissible speeds} \label{SS:surfacesconvexity}

Consider the function $G:=-\frac{\kappa_1}{\kappa_2}$. Arguing as in Lemma \ref{lem:evolve G surfaces} yields
\[
\frac{(\partial_t-\mathcal{L})G}{G}\le -\gamma\frac{\vert\nabla A\vert^2}{F^2}+\gamma^{-1}\frac{\vert \nabla G\vert^2}{G^2}
\]
for some $\gamma=\gamma(f,\alpha,\varepsilon)>0$ wherever $-\alpha^{-1}\le\frac{\kappa_1}{\kappa_2}\le-\varepsilon$.

Thus, if we set, for any $\varepsilon>0$ and $\sigma\in(0,1),$ $G_{\varepsilon,\sigma}:=(G-\varepsilon)F^\sigma$, then, arguing as above, we find that
\begin{align*}
\frac{(\partial_t-\mathcal{L})G_{\varepsilon,\sigma}}{G_{\varepsilon,\sigma}}\le{}&\sigma\vert A\vert^2_F-\gamma \frac{\vert\nabla A\vert^2}{F^2}+\gamma^{-1}\frac{\vert\nabla G_{\varepsilon,\sigma}\vert^2}{G_{\varepsilon,\sigma}^2}
\end{align*}
for some $\gamma=\gamma(f,\alpha,\varepsilon)>0$, so long as $\sigma\le\ell=\ell(f,\alpha,\varepsilon)$.

We may thus proceed exactly as in Section \ref{sec:convexity estimate for convex speeds} (from \eqref{eq:key estimate for ev Ges} onwards) to establish Theorem \ref{thm:convexity surfaces}.

\section{Hypersurfaces evolving by concave speeds} \label{S:concave}

Consider now the evolution of hypersurfaces by concave speeds $f:\Gamma^n\to\R$. In contrast to convex speeds, concave speeds do not tend to preserve general convex cones of curvatures, but do preserve the cones\footnote{To see this, apply the maximum principle to \eqref{eq:evlnGes} when $\sigma=0$ with $\varepsilon=\alpha^{-1}$.} $\Gamma_0^n:= \{z\in\R^n:f(z_1,\dots,z_n)\ge \alpha \, g(z_1,\dots,z_n)\}$, $\alpha>0$, whenever $g$ is a monotone non-decreasing, one-homogeneous, convex symmetric function which is defined on a convex cone containing $\Gamma^n$ \cite{MR1385524}. If $\Gamma^n\Subset\{z\in\R^n:g(z_1,\dots,z_n)>0\}$ and $f|_{\partial\Gamma^n}\equiv 0$, then these cones satisfy $\Gamma^n_0\Subset \Gamma^n$, which ensures that the flow remains uniformly parabolic \cite{MR1385524}. We will establish the following cylindrical estimates for such speeds.

\begin{theorem}[Cylindrical estimates --- concave speeds]\label{thm:cylindrical estimates concave}
Let $f:\Gamma^n\to\R$ be a concave admissible speed such that $\Gamma^n\subset\Gamma^n_{m+1}$, $m\le n-1$, and $f|_{\partial\Gamma^n}=0$. If a solution to the flow with speed $F=f(\kappa_1,\dots,\kappa_n)$ is properly defined in the spacetime cylinder $\overline B_{\lambda R} \times \left[ 0, \frac{1}{2n} R^2 \right) \subset \mathbb{R}^{n+1} \times \mathbb{R}$ and satisfies
\begin{enumerate}
\item $\displaystyle\inf_{B_{\lambda R} \times \left\{ 0 \right\} \cup \partial B_{\lambda R} \times \left( 0, \frac{1}{2n} R^2 \right)} \frac{F}{H+\vert A\vert} \geq \alpha > 0$,
\item $\displaystyle\sup_{B_{\lambda R} \times \left\{ 0 \right\} \cup B_{\lambda R}\setminus B_{(\lambda-\delta)R} \times \left( 0, \frac{1}{2n} R^2 \right)} F \leq \Theta\, R^{-1}$, and
\item $\displaystyle\frac{R^2}{2n} \int_{B_{\lambda R}} F^q d\mu_0 + \frac{1}{\delta^2} \int_0^{\frac{R^2}{2n}}\!\!\!\int_{B_{\lambda R} \backslash B_{(\lambda-\delta)R}} F^q d\mu\, dt \leq \left( \Theta\, R^{-1}\right)^q \left( V \, R\right)^{n+2}$,
\end{enumerate}
then, given any $\varepsilon>0$ and $\vartheta \in \left(0,1\right)$, it satisfies
\begin{equation}\label{eq:cylindrical estimates concave}
\kappa_n-c_mF\leq \varepsilon\,F + C_\varepsilon R^{-1}\;\;\text{in}\;\; B_{\left( \lambda- \delta\right) \vartheta R} \times \left( 0, \tfrac{1}{2n} R^2 \right)\,,
\end{equation}
where $C_\varepsilon= c\left(n,f,\alpha,q,\varepsilon\right) \Theta \left( \frac{\Theta\, V}{1-\vartheta} \right)^{\frac{2}{q}}$ and $c_m^{-1}:= f(\underbrace{0,\dots,0}_{m\text{-times}},1,\dots,1)$.
\end{theorem}

We refer to \eqref{eq:cylindrical estimates concave} as ``cylindrical estimates'' since (under the conditions imposed on $f$) a weakly locally convex hypersurface with $F>0$ satisfying
\[
\kappa_n\le c_mF
\]
must satisfy either
\[
\kappa_1+\dots+\kappa_m>0\;\;\text{or else}\;\;\kappa_1=\dots=\kappa_m=0\;\;\text{and}\;\;\kappa_{m+1}=\dots=\kappa_n>0\,.
\]
Note that the $m=0$ case gives an umbilic estimate.

\subsection{Proof of the cylindrical estimates for flows by concave speeds} 

Given $m\in\{0,\dots,n-1\}$, consider the flow by any concave admissible speed function $f:\Gamma^n\to\mathbb{R}$ with $f|_{\partial\Gamma^n}=0$ such that $\Gamma^n\subset\Gamma_{m+1}^n$. 
Along any solution to this flow, we define the pinching function $G:M^n\times[0,T)\to\mathbb{R}^{n+1}$ as in \cite[Section 3]{MR4129354}. This function is a smooth, one-homogeneous, symmetric function of the principal curvatures which is non-negative in $\Gamma_{m+1}^n$ and vanishes precisely when $\kappa_n\le c_m F$, where
\[
c_m^{-1}:=f(\underbrace{0,\dots,0}_{m\text{-times}},1,\dots,1)\,.
\]
It was shown in \cite[Section 3]{MR4129354} that, for any $\sigma\in(0,1)$ and $\varepsilon>0$, the function $G_{\varepsilon,\sigma}:=(G-\varepsilon F)F^{\sigma-1}$ satisfies
\[
(\partial_t-\mathcal{L})G_{\varepsilon,\sigma}\le G_{\varepsilon,\sigma}\left(\sigma\vert A\vert^2_F-\gamma\frac{\vert\nabla A\vert^2}{F^2}+\gamma^{-1}\frac{\vert\nabla G_{\varepsilon,\sigma}\vert^2}{G_{\varepsilon,\sigma}^2}\right)
\]
for some $\gamma=\gamma(n,f,\alpha,\varepsilon)>0$ wherever $G_{\varepsilon,\sigma}>0$. 
We may thus proceed exactly as in Section \ref{sec:convexity estimate for convex speeds} to establish Theorem \ref{thm:cylindrical estimates concave}.

\section{Convex hypersurfaces evolving by concave, inverse-concave speeds} \label{S:Fcic}

In the previous section, we established $m$-cylindrical estimates for concave speeds $f:\Gamma^n\to\R$ which satisfy the additional condition $\Gamma^n\subset\Gamma_{m+1}^n$ and $f|_{\partial\Gamma^n}=0$ (Theorem \ref{thm:cylindrical estimates concave}). The latter condition was only needed to ensure that the flow admits a preserved cone $\Gamma_0^n\Subset \Gamma^n$ of curvatures. In case $m=0$ (i.e. $\Gamma^n=\Gamma^n_+$), this is also guaranteed by inverse-concavity of $f$. Indeed, Andrews showed, using a sharpened tensor maximum principle, that positive lower bounds for $\kappa_1/H$ are preserved under such flows\footnote{Andrews' argument was presented for compact hypersurfaces satisfying a global positive lower bound for the ratio $\frac{\kappa_1}{H}$. But, of course, the argument still shows that positive lower bounds on the parabolic boundary of a region in which the flow is properly defined propagate to its interior.} \cite{MR2339467}. We have therefore established the following umbilic estimate.

\begin{theorem}[Umbilic estimate --- concave, inverse-concave speeds]\label{thm:umbilic inverse-concave}
Let $f:\Gamma_+^n\to\R$ be an admissible speed function which is both concave and inverse-concave. If a solution to the flow with speed $F=f(\vec\kappa)$ is properly defined in the spacetime cylinder $\overline B_{\lambda R} \times \left[0,\frac{1}{2n}R^2\right) \subset \mathbb{R}^{n+1}\times\mathbb{R}$ and satisfies
\begin{enumerate}
\item $\displaystyle\inf_{B_{\lambda R} \times \left\{ 0 \right\} \cup \partial B_{\lambda R} \times \left(0,\frac{1}{2n}R^2\right)}\frac{\kappa_1}{H}\geq\alpha>0$,
\item $\displaystyle\sup_{B_{\lambda R} \times \left\{ 0 \right\} \cup B_{\lambda R}\setminus B_{(\lambda-\delta)R} \times \left( 0, \frac{1}{2n} R^2 \right)} F \leq \Theta\, R^{-1}$, and
\item $\displaystyle\frac{R^2}{2n} \int_{B_{\lambda R}} F^q d\mu_0 + \frac{1}{\delta^2} \int_0^{\frac{R^2}{2n}}\!\!\!\int_{B_{\lambda R} \backslash B_{(\lambda-\delta)R}} F^q d\mu\, dt \leq \left( \Theta\, R^{-1}\right)^q \left( V \, R\right)^{n+2}$,
\end{enumerate}
then, given any $\varepsilon>0$ and $\vartheta \in \left(0,1\right)$, it satisfies
\begin{equation}\label{eq:cylindrical inverse-concave speeds}
\kappa_1-c_0F\geq -\varepsilon\,F - C_\varepsilon R^{-1}\;\;\text{in}\;\; B_{\left( \lambda- \delta\right) \vartheta R} \times \left( 0, \tfrac{1}{2n} R^2 \right)\,,
\end{equation}
where $C_\varepsilon= c\left(n,f,\alpha,q,\varepsilon \right) \Theta \left( \frac{\Theta\, V}{1-\vartheta} \right)^{\frac{2}{q}}$ and $c_0^{-1} := f(1,\dots,1)$.
\end{theorem}

\section{Hypersurfaces evolving by Lynch's speeds}

Flows by concave speeds do not in general tend to preserve (or improve) lower bounds for the second fundamental form. As such, obtaining a general convexity estimate for flows by concave speeds remained until recently an important open problem. In a major breakthrough, Lynch \cite{MR4484205} discovered a family of concave speeds which admit both cylindrical and convexity estimates. These speeds are a nonlinear interpolation of concave speeds with the mean curvature. More precisely, given an admissible speed function $f_1:\Gamma^n\to\mathbb{R}$ with $f_1|_{\partial\Gamma^n}=0$ which is strictly concave in non-radial directions, Lynch considered, for $\rho\in[0,1]$, the nonlinear interpolation
\begin{equation}\label{eq:Lynch's speeds}
f_\rho(z_1,\dots,z_n):=\left(\frac{\rho}{f_1(z_1,\dots,z_n)}+\frac{1-\rho}{f_0(z_1,\dots,z_n)}\right)^{-1}
\end{equation}
of $f_1$ with ($n$-times) the arithmetic mean,
\[
f_0(z_1,\dots,z_n):=z_1+\dots+z_n\,.
\]
Observe that the speeds $f_\rho$ are monotone in $\rho$ \cite{MR4484205}.
\begin{lemma}
For any $z\in\Gamma^n$, the function $\rho\mapsto f_\rho(z)$ is non-increasing. 
\end{lemma}
\begin{proof}
Since each $z\in\Gamma^n$ may be written as $z=\lambda\vec 1+v$ for some $v\in \vec 1{}^\perp=\{z\in\R^n:z_1+\dots+z_n=0\}$, where $\vec 1:=(1,\dots,1)$, the concavity and normalization of $f$ imply that
\[
f_1(z)=f_\rho(\lambda \vec 1+ v)\le f_1(\lambda\vec1)=n\lambda=f_0(z)\,.
\]
Thus,
\begin{equation}\label{eq:frho monotone in rho}
\frac{d}{d\rho}f_\rho=\frac{f_\rho^2}{f_0f_1}(f_1-f_0)\le 0\,.\qedhere
\end{equation}
\end{proof}

It is easy to see that each $f_\rho:\Gamma^n\to\mathbb{R}$ is itself a concave admissible speed function which vanishes on $\Gamma^n$; in particular, this means that solutions to the corresponding flow admit an $m$-cylindrical estimate if $\Gamma^n\subset\Gamma_{m+1}^n$, due to Theorem \ref{thm:cylindrical estimates concave}. We will need the following slightly different local cylindrical estimate, however.

\begin{theorem}[Cylindrical estimate --- Lynch's speeds]\label{thm:cylindrical estimates strictly concave}
Let $f_1:\Gamma^n\to\R$ be an admissible speed which is strictly concave in non-radial directions, with $\Gamma^n\subset\Gamma_{m+1}^n$, $2\le m\le n-1$, and $f_1|_{\partial\Gamma^n}=0$, and define $f_\rho$ as in \eqref{eq:Lynch's speeds}. If a solution to the flow with speed $F_\rho=f_\rho(\kappa_1,\dots,\kappa_n)$ is properly defined in the spacetime cylinder $\overline B_{\lambda R} \times \left[0,\frac{1}{2n}R^2\right)\subset \mathbb{R}^{n+1} \times \mathbb{R}$ and satisfies
\begin{enumerate}
\item $\displaystyle\inf_{B_{\lambda R} \times \left\{ 0 \right\} \cup \partial B_{\lambda R} \times \left( 0, \frac{1}{2n} R^2 \right)} \frac{F_\rho}{H} \geq \alpha > 0$,
\item $\displaystyle\sup_{B_{\lambda R} \times \left\{ 0 \right\} \cup B_{\lambda R}\setminus B_{(\lambda-\delta)R} \times \left( 0, \frac{1}{2n} R^2 \right)} F_\rho \leq \Theta\, R^{-1}$, and
\item $\displaystyle\frac{R^2}{2n}\int_{B_{\lambda R}}d\mu_0 + \frac{1}{\delta^2}\int_0^{\frac{R^2}{2n}}\!\!\!\int_{B_{\lambda R} \backslash B_{(\lambda-\delta)R}}d\mu\, dt \leq  \left(V R\right)^{n+2}$,
\end{enumerate}
then, given any $\varepsilon>0$ and $\vartheta \in \left(0,1\right)$, it satisfies
\begin{equation}\label{eq:cylindrical estimates strictly concave}
\tfrac{1}{n-m}H-c_mF_\rho\leq \varepsilon\,F_\rho+C_\varepsilon R^{-1}\;\;\text{in}\;\; B_{(\lambda- \delta)\vartheta R} \times \left( 0, \tfrac{1}{2n} R^2 \right)\,,
\end{equation}
where $c_m^{-1}:=f_\rho(\underbrace{0,\dots,0}_{m\text{-times}},1,\dots,1)$ and $C_\varepsilon=C_\varepsilon(n,f,\alpha,\Theta,V,\varepsilon)$.
\end{theorem}

Remarkably, Lynch was able to show that the speeds defined by \eqref{eq:Lynch's speeds} also admit a convexity estimate if $\rho$ is sufficiently small. We will establish the following localization of Lynch's convexity estimate.\footnote{We only need to consider the case $m\ge2$ since the cylindrical estimate \eqref{eq:cylindrical estimates concave} already implies a convexity estimate when $m=1$. This is because the homogeneity and monotonicity of $f$ imply that
\[
\kappa_n-c_1F=c_1\big(\kappa_nf(0,1,\dots,1)-f(\kappa_1,\dots,\kappa_n)\big)\ge c_1\big(f(0,\kappa_n,\dots,\kappa_n)-f(\kappa_1,\kappa_n,\dots,\kappa_n)\big)\,,
\]
and hence (by monotonicity of $f$) $\kappa_n-c_1F\le 0$ only if $\kappa_1\ge 0$. Homogeneity then guarantees that $-\kappa_1\le c(\kappa_n-c_1F)$ wherever $-\kappa_1\ge \varepsilon F$ for some $c=c(n,f,\alpha,\varepsilon)$.}

\begin{theorem}[Convexity estimate --- Lynch's speeds]\label{thm:convexity estimate concave}
Let $f_1:\Gamma^n\to\R$ be an admissible speed which is strictly concave in non-radial directions, with $\Gamma^n\subset\Gamma_{m+1}^n$, $2\le m\le n-1$, and $f_1|_{\partial\Gamma^n}=0$, and define $f_\rho$ as in \eqref{eq:Lynch's speeds}. If $\rho\le\rho_0(n,f_1)$ and a solution to the flow with speed $F_\rho=f_\rho(\kappa_1,\dots,\kappa_n)$ is properly defined in the spacetime cylinder $\overline B_{\lambda R} \times \left[ 0, \frac{1}{2n} R^2 \right) \subset \mathbb{R}^{n+1} \times \mathbb{R}$ and satisfies
\begin{enumerate}
\item $\displaystyle\inf_{B_{\lambda R} \times \left\{ 0 \right\} \cup \partial B_{\lambda R} \times \left( 0, \frac{1}{2n} R^2 \right)} \frac{F_\rho}{H} \geq \alpha > 0$,
\item $\displaystyle\sup_{B_{\lambda R} \times \left\{ 0 \right\} \cup B_{\lambda R}\setminus B_{(\lambda-\delta)R} \times \left( 0, \frac{1}{2n} R^2 \right)} F_\rho \leq \Theta\, R^{-1}$, and
\item $\displaystyle\frac{R^2}{2n}\int_{B_{\lambda R}}d\mu_0 + \frac{1}{\delta^2}\int_0^{\frac{R^2}{2n}}\!\!\!\int_{B_{\lambda R} \backslash B_{(\lambda-\delta)R}}d\mu\, dt \leq  \left(V R\right)^{n+2}$,
\end{enumerate}
then, given any $\varepsilon>0$ and $\vartheta \in \left(0,1\right)$, it satisfies
\begin{equation}\label{eq:convexity estimate concave}
\kappa_1\geq -\varepsilon\,F_\rho-C_\varepsilon R^{-1}\;\;\text{in}\;\; B_{(\lambda- \delta)\vartheta R} \times \left( 0, \tfrac{1}{2n} R^2 \right)\,,
\end{equation}
where $C_\varepsilon= %
C_\varepsilon(n,f_1,\rho,\alpha,\Theta,V,\varepsilon)$.
\end{theorem}

\subsection{Proof of the cylindrical estimate for Lynch's flow}

Define a function $G:M^n\times I\to\R$ by $G:=\frac{1}{n-m}H-c_mF_\rho$ and set, for any $\varepsilon>0$ and $\sigma\in (0,1)$, $G_{\varepsilon,\sigma}:=(G-\varepsilon F_\rho)F_\rho^{\sigma-1}$. By \cite[Lemma 4.2]{MR4484205},
\[
(\partial_t-\mathcal{L})G_{\varepsilon,\sigma}\le G_{\varepsilon,\sigma}\left(\sigma\vert A\vert^2_{F_\rho}-\gamma\frac{\vert\nabla A\vert^2}{F_\rho^2}+\gamma^{-1}\frac{\vert\nabla G_{\varepsilon,\sigma}\vert^2}{G_{\varepsilon,\sigma}^2}\right)
\]
for some $\gamma=\gamma(n,f,\alpha,\varepsilon)>0$ wherever $G_{\varepsilon,\sigma}>0$. 
We may thus proceed exactly as in Section \ref{sec:convexity estimate for convex speeds} (from \eqref{eq:key estimate for ev Ges} onwards) to establish Theorem \ref{thm:cylindrical convex}.

\subsection{Proof of the convexity estimate for Lynch's flow}

Assume (without loss of generality) that $R=1$ (we shall reintroduce the variable $R$ in the definition of our cut-off functions later) and introduce constants $\delta_1,\delta_2<\delta$ and $\vartheta_1,\vartheta_2<1$ such that
\[
(\lambda-\delta)\vartheta=((\lambda-\delta_1)\vartheta_1-\delta_2)\vartheta_2\,.
\]
One way to do this, which will suffice for our purposes, is to take
\[
\vartheta_1=\vartheta_2=\vartheta^\sharp:=\sqrt{\vartheta}\;\;\text{and}\;\;\delta_1=\delta_2=\delta^\sharp:=\frac{\delta\vartheta^\sharp}{1+\vartheta^\sharp}\,.
\]
This guarantees that
\[
\frac{\delta}{\delta_i}=1+\vartheta^{-\frac{1}{2}}
\]
depends only on $\vartheta$.

Observe that any solution to the flow by speed $F_\rho$ satisfying the hypotheses of Theorem \ref{thm:convexity estimate concave} satisfies the hypotheses of Theorem \ref{thm:cylindrical estimates strictly concave} with $q=2$ (say), $\delta$ replaced by $\delta_1$, and $V$ replaced by $V_1:=(\frac{\delta}{\delta_1})^{\frac{2}{n+2}}V$, and hence enjoys the cylindrical estimate \eqref{eq:cylindrical estimates strictly concave} 
\[
\tfrac{1}{n-m}H\le (c_m+\varepsilon)F_\rho+C_\varepsilon
\]
in $B_{(\lambda-\delta_1)\theta_1}\times(0,\frac{1}{2n})$ with $c_m^{-1}:=f_\rho(\underbrace{0,\dots,0}_{m\text{-times}},1,\dots,1)$ $C_\varepsilon=C_\varepsilon(n,f_1,\rho,\alpha,\Theta,V_1,\vartheta_1,\varepsilon)=C_\varepsilon(n,f_1,\rho,\alpha,\Theta,V,\vartheta,\varepsilon)$. In fact, since
\[
\frac{f_0}{f_\rho}=\rho\frac{f_0}{f_1}+1-\rho\,,
\]
we may actually estimate, for any $\varepsilon>0$,
\[
H\le (c_m+\varepsilon)F_1+C_\varepsilon
\]
with now $c_m^{-1}:=f_1(\underbrace{0,\dots,0}_{m\text{-times}},1,\dots,1)$ and $C_\varepsilon$ of the same form.
Thus, if we set $\mu:=\frac{1}{2(1+10^{-10})(n-m)c_m}$, then, since $f_\rho\ge f_1$, we can find some $K=K(n,f_1,\rho,\alpha,\Theta,V,\vartheta)$ such that
\[
F_\rho\ge F_1 \ge 2\mu H-K\,.
\]

Following Lynch \cite[Section 5.1]{MR4484205}, define the pinching function
\[
G_\varepsilon:=\frac{-\kappa_1-\varepsilon F_\rho}{F_\rho-\mu H+K}\,.
\]
Given any $\Gamma_0\Subset\Gamma$, Lynch showed that 
\begin{align*}
(\partial_t-\mathcal{L})G_{\varepsilon}\le{}& C\left(K\vert A\vert G_\varepsilon+(\rho^{-1}+KH^{-1})\frac{\vert\nabla G_\varepsilon\vert^2}{G_\varepsilon}\right)-\rho C^{-1} G_\varepsilon\frac{\vert\nabla A\vert^2}{(F_\rho-\mu H+K)H}\\
{}&-(C^{-1}-C\rho)\sum_{p=1}^n\sum_{q=2}^n\frac{\vert\nabla_1A_{pq}\vert^2}{(F_\rho-\mu H+K)H}
\end{align*}
in the distributional sense wherever $G_\varepsilon>0$ and $\vec\kappa\in\Gamma_0$, where $C=C(n,f_1,\Gamma_0)$ \cite[Proposition 6.1]{MR4484205}. Observe that 
\[
\vec\kappa\in \Gamma_0:=\left\{z\in\Gamma:\frac{f_0(z)}{f_1(z)}\le\frac{1}{2\mu}+1\right\}
\]
wherever $F_\rho\ge K$. Thus, if we impose $\rho\le C^{-2}$, then\footnote{To carry out this step, it is crucial that the constant $\mu$ in the definition of $G_\varepsilon$ does not depend on $\rho$.}
\[
(\partial_t-\mathcal{L})G_{\varepsilon}\le C\left(K\vert A\vert G_\varepsilon+\rho^{-1}\frac{\vert\nabla G_\varepsilon\vert^2}{G_\varepsilon}\right)-\rho C^{-1} G_\varepsilon\frac{\vert\nabla A\vert^2}{H^2}
\]
wherever $G_\varepsilon>0$ and $F_\rho\ge K$, where $C=C(n,f_1)$. Consider now, for $\sigma\in(0,1)$, the function $G_{\varepsilon,\sigma}:=G_{\varepsilon}F^\sigma$. Using the identity
\[
\frac{\nabla G_{\varepsilon,\sigma}}{G_{\varepsilon,\sigma}}=\frac{\nabla G_{\varepsilon}}{G_\varepsilon}+\sigma\frac{\nabla F_\rho}{F_\rho}
\]
to estimate
\[
\frac{\vert\nabla G_{\varepsilon}\vert^2}{G_{\varepsilon}^2}\le 2\frac{\vert\nabla G_{\varepsilon,\sigma}\vert^2}{G_{\varepsilon,\sigma}^2}+2\sigma^2\frac{\vert\nabla F_\rho\vert^2}{F_\rho^2}\,,
\]
we find, for $\sigma\le\ell(n,f_1,\rho)$, that
\begin{align*}
\frac{(\partial_t-\mathcal{L})G_{\varepsilon,\sigma}}{G_{\varepsilon,\sigma}}
\le{}&C\left(K\vert A\vert+\rho^{-1}\frac{\vert\nabla G_\varepsilon\vert^2}{G_\varepsilon^2}\right)-\rho C^{-1}\frac{\vert\nabla A\vert^2}{H^2}\\
{}&+\sigma\vert A\vert_F^2+\sigma(1-\sigma)\frac{\vert\nabla F_\rho\vert_{F_\rho}^2}{F_\rho^2}-2\sigma\left<\frac{\nabla G_{\varepsilon,\sigma}}{G_{\varepsilon,\sigma}},\frac{\nabla F_\rho}{F_\rho}\right>_{F_\rho}\\
\le{}&C\left(\left(\frac{K}{F_\rho}+\sigma\right)\vert A\vert^2+\rho^{-1}\frac{\vert\nabla G_{\varepsilon,\sigma}\vert^2}{G_{\varepsilon,\sigma}^2}\right)-\rho C^{-1}\frac{\vert\nabla A\vert^2}{H^2}
\end{align*}
in the distributional sense wherever $G_{\varepsilon,\sigma}>0$ and $F_\rho\ge K$, where $C=C(n,f_1)$.

Now, Young's inequality implies that
\[
F_\rho^\sigma\le 1-\sigma+\sigma F_\rho\,.
\]
Thus, since
\[
G_\varepsilon\le \frac{-\kappa_1}{\mu H}\le C(n,f_1,\alpha)\,,
\]
we may estimate, wherever $G_{\varepsilon,\sigma}\ge k$,
\[
k\le G_{\varepsilon,\sigma}=G_\varepsilon F_\rho^\sigma\le C(1+\sigma F_\rho)\,.
\]
We conclude that
\[
\frac{K}{F_\rho}\le \frac{\sigma CK}{k-C}\le \sigma
\]
wherever $G_{\varepsilon,\sigma}\ge k$ if $k\ge k_0=C(K+1)$, where $C=C(n,f_1,\alpha)$.

Consider, for $k\ge k_0$ and $p\ge 10$,
\[
v_k\doteqdot (G_{\varepsilon,\sigma}-k)_+^{\frac{p}{2}}\,.
\]
Proceeding as in \eqref{E:vk} (cf. \cite{LynchThesis}), we find that
\begin{align}\label{eq:L2 to Linfty concave speeds}
\frac{d}{dt}\int \psi^2v_k^2\,d\mu+C^{-1}\int{}&\left(\frac{\vert\nabla v_k\vert^2}{v_k^2}+p\rho\frac{\vert\nabla A\vert^2}{F^2_\rho}+H^2\right)\psi^2v_k^2\,d\mu\nonumber\\
\le{}&\frac{C}{(R-r)^2}\int_{B_R\setminus B_r}v_k^2\,d\mu+C\sigma p\int G_{\varepsilon,\sigma}(G_{\varepsilon,\sigma}-k)_+^{p-1}\vert A\vert^2\psi^2\,d\mu\,,
\end{align}
where $C=C(n,f_1)$, so long as $p\ge L=L(n,f_1,\rho)$, $\sigma\le \ell=\ell(n,f_1,\rho)$ and $k\ge k_0=k_0(n,f_1,\rho,\alpha,\Theta,V,\vartheta)$. Estimating
\begin{align*}
G_{\varepsilon,\sigma}(G_{\varepsilon,\sigma}-k)_+^{p-1}={}&(G_{\varepsilon,\sigma}-k)_+^p+k(G_{\varepsilon,\sigma}-k)_+^{p-1}\\
\le{}& v_k^2+\tfrac{1}{p}k^p+(1-\tfrac{1}{p})v_k^2
\end{align*}
and $\vert A\vert^2\le CH^2$, $C=C(n,f_1)$ (wherever $v_k>0$), we obtain
\begin{align*}
\frac{d}{dt}\int \psi^2v_k^2\,d\mu+C^{-1}\int{}&\left(\frac{\vert\nabla v_k\vert^2}{v_k^2}+p\rho\frac{\vert\nabla A\vert^2}{F^2_\rho}+H^2\right)\psi^2v_k^2\,d\mu\\
\le{}&\frac{C}{(R-r)^2}\int_{B_R\setminus B_r}v_k^2\,d\mu+2\sigma pC\int \psi^2v_k^2\vert A\vert^2\,d\mu+C\sigma k^p\int H^2\psi^2\,d\mu\,.
\end{align*}
Absorbing the penultimate term using the Poincar\'e-type inequality (Proposition \ref{prop:Poincare}) as before now yields
\[
\frac{d}{dt}\int \psi^2v_k^2\,d\mu\le\frac{C}{(R-r)^2}\int_{B_R\setminus B_r}v_k^2\,d\mu+C\sigma k^p\int H^2\psi^2\,d\mu
\]
and hence, by Young's inequality,
\begin{align*}
\frac{d}{dt}\int(v_k^2+C\sigma k^p)\psi^2\,d\mu
\le{}&\frac{C(1+\sigma k^p)}{(R-r)^2}\int_{B_R\setminus B_r} v_k^2\psi^2\,d\mu
\end{align*}
for some (much larger) $C=C(n,f_1)$, so long as $p\ge L=L(n,f_1,\rho,\varepsilon)$, $\sigma p^{-\frac{1}{2}}\ge \ell=\ell(n,f_1,\rho,\varepsilon)$ and $k\ge k_0=k_0(n,f_1,\rho,\alpha,\Theta,V,\vartheta)$. Integrating with respect to time yields
\[
\sup_{t\in[0,\frac{1}{2n})}\int(v_k^2+C\sigma k^p)\psi^2\,d\mu\le \int(v_k^2+C\sigma k^p)\psi^2\,d\mu_0+\frac{C(1+\sigma k^p)}{(R-r)^2}\int_0^{\frac{1}{2n}}\!\!\!\int_{B_R\setminus B_r}v_k^2\,d\mu\,dt
\]
Integrating again then yields, upon setting $R=\lambda_2:=(\lambda-\delta_1)\theta_1$ and $r=\lambda_2-\delta_2$, the following replacement for Lemma \ref{lem:L2 convex speeds}.
\begin{lemma}[$L^2$-estimate]
If $p\ge L=L(n,f_1,\rho,\alpha,\varepsilon)$, $\sigma p^{\frac{1}{2}}\le \ell=\ell(n,f_1,\rho,\alpha,\varepsilon)$ and $k\ge k_0=k_0(n,f_1,\rho,\alpha,\Theta,V,\vartheta)$, then
\begin{align}\label{eq:L2 estimate concave speeds}
\int\!\!\!\int_{B_{\lambda_2-\delta_2}}v_{k_0}^2\,d\mu\le{}& \frac{1}{2n}\int_{B_{\lambda_2}}(v_{k_0}^2+C\sigma k_0^p)\,d\mu_0+\frac{C(1+\sigma k_0^p)}{\delta_2^2}\int\!\!\!\int_{B_{\lambda_2}\setminus B_{\lambda_2-\delta_2}}v_{k_0}^2\,d\mu\,dt\,,
\end{align}
where $C=C(n,f_1)$.
\end{lemma}
Note that the right hand side of \eqref{eq:L2 estimate concave speeds} may be estimated in terms of the data $n$, $f_1$, $\rho$, $\alpha$, $\Theta$, $V$, $\vartheta$, $\sigma$ and $p$.

\subsubsection{From $L^2$ to $L^\infty$}

If $k>2k_0$, then we can estimate $G_{\varepsilon,\sigma}-k_0\ge k_0$ 
and hence
\[
G_{\varepsilon,\sigma}(G_{\varepsilon,\sigma}-k)_+^{p-1}=(G_{\varepsilon,\sigma}-k_0)(G_{\varepsilon,\sigma}-k)_+^p+k_0(G_{\varepsilon,\sigma}-k)_+^p\le 2v_{k_0}^2
\]
wherever $v_k^2>0$. Since $k\le CF_\rho^\sigma$, $C=C(n,f_1,\alpha)$, wherever $v_k>0$ and, without loss of generality, $2k_0\ge C$, we also have
\[
v_{k_0}^2F_\rho^2=\left(G_\varepsilon F_\rho^{\sigma+\frac{2}{p}}-k_0F_\rho^{\frac{2}{p}}\right)_+^p\le \left(G_{\varepsilon,\sigma'}-k_0\left(\frac{k}{C}\right)^{\frac{2}{\sigma p}}\right)_+^p\le \left(G_{\varepsilon,\sigma'}-k_0\right)_+^p\,,
\]
where $\sigma':=\sigma+\frac{2}{p}$. Recalling \eqref{eq:L2 to Linfty concave speeds}, we therefore have, taking $k_0$ larger as necessary, the following substitute for \eqref{E:vk}:
\begin{align*}
\frac{d}{dt}\int \psi^2v_k^2\,d\mu+C^{-1}\int{}&\left(\vert\nabla v_k\vert^2+v_k^2H^2\right)\psi^2\,d\mu\nonumber\\
\le{}&\frac{C}{\delta_2^2}\int_{B_{\lambda_2}\setminus B_{\lambda_2-\delta_2}}v_k^2\,d\mu+C\sigma p\int_{B_{\lambda_2}}\left(G_{\varepsilon,\sigma'}-k_0\right)_+^p\,d\mu
\end{align*}
whenever $k>k_0=k_0(n,f_1,\rho,\alpha,\Theta,V,\vartheta)$. 

We may now employ the H\"older and interpolation inequalities, the Michael--Simon Sobolev inequality, the $L^2$-estimate \eqref{eq:L2 estimate concave speeds}, and Stampacchia's lemma in the same manner as in Section \ref{sec:convexity estimate for convex speeds} (with $G_{\varepsilon,\sigma'}-k_0$ playing the role of $G_{\varepsilon,\sigma'}$) to establish the $L^\infty$ estimate.

\section{Extending the estimates to Riemannian ambient spaces} \label{S:Riem}

As has been observed by Huisken \cite{MR837523} and others \cite{MR3714512,MR892052,LangfordNguyen,MR4484205,Ng15}, the key estimates of the Huisken--Stampacchia iteration scheme may also be carried out for flows in Riemannian ambient spaces. The basic observation is that the Gauss--Codazzi equations provide the estimates
\begin{align*}
\left\vert \nabla_tA-\dot F^{ij}\nabla_i\nabla_jA-\dot F^{ij}A^2_{ij}A-\ddot F^{pq,rs}\nabla A_{pq}\nabla A_{rs}\right\vert\le C\left(\vert A\vert^2+1\right)\;\;\text{and}\;\;\big\vert \widehat{\nabla A}\big\vert\le C\,,
\end{align*}
where $\widehat{\nabla A}$ is the skew part of $\nabla A$ and the constants depend only on bounds for the ambient curvature and its derivative. These estimates ensure that, so long as the ambient curvature and its covariant derivative are bounded, the ambient geometry only contributes lower order terms, which are easily absorbed. We note that these considerations also apply, with minor modification, to flows by functions of \emph{shifted} principal curvatures (as in \cite{MR1267897,MR3714512,MR4484205}). 

Note, however, that the scheme also requires the preservation of the intermediate convexity condition (assumed to hold on the parabolic boundary). This need \emph{not} hold in general, as ambient terms in the evolution equation for the second fundamental form typically preclude the application of the tensor maximum principle. One important setting in which a suitable version of the intermediate convexity condition \emph{is} preserved concerns the flow by suitable concave functions of suitably shifted principal curvatures \cite{MR1267897,MR3714512,LynchThesis,MR4484205}; see, in particular, \cite[Theorem 7.1]{MR4484205}. (Note that uniform ellipticity may still be established as a consequence of the preserved curvature condition, though in a less direct way than in the Euclidean setting; see \cite[Section 4]{MR1267897} or \cite[Theorem 7.1]{MR4484205}). 

Finally, in order to establish \emph{localized} estimates, one also requires the introduction of suitable cut-off functions. This poses no difficulty at sufficiently small scales (relative to bounds for the ambient geometry). At large scales, suitable cut-off functions may be constructed under reasonable growth conditions on the ambient metric (e.g. asymptotic flatness or co-compactness). 

We refrain from attempting a general statement here, as such a statement would be too cumbersome to be of much utility, and instead defer such statements to future work (where specific applications will be considered).

\bibliographystyle{plain}
\bibliography{bibliography}

\end{document}